\documentclass[a4paper,11pt,english]{article}

\usepackage{amsmath,amsfonts,amssymb,amsthm}
\usepackage[all]{xy}
\usepackage{geometry,babel}

\newtheorem {thm}{Theorem}
\newtheorem* {thm*}{Theorem}

\newtheorem {cor}[thm]{Corollary}

\newtheorem* {cor*}{Corollary}
\newtheorem {lem}[thm]{Lemma}
\newtheorem {prop}[thm]{Proposition}

\newtheorem {rem}[thm]{Remark}

\theoremstyle{definition}

\newtheorem* {conj*}{Conjecture}

\newtheorem* {quest*}{Question}

\numberwithin{thm}{section}

\DeclareMathOperator{\End}{End}
\DeclareMathOperator{\ord}{ord}

\DeclareMathOperator{\Hom}{Hom}

\DeclareMathOperator{\rad}{rad}

\newcommand{\E}{\mathrm{E}}

\newcommand{\G}{\Gamma}

\newcommand{\Z}{\mathbb{Z}}

\renewcommand{\H}{\mathrm{H}}

\newcommand{\e}{\varepsilon}
\newcommand{\p}{\mathfrak{p}}

\newcommand{\ve}{\varepsilon}

\renewcommand{\a}{\alpha}

\newcommand{\zs}{{\{0\}}}

\author{Chris Hall and Antonella Perucca}
\title{Characterizing Abelian Varieties by the Reductions of the Mordell-Weil Group} 
\date{}

\begin{document}

\maketitle

\begin{abstract}
Let $A$ be an abelian variety defined over a number field $K$. If $\p$ is a prime of $K$ of good reduction for $A$, let $A(K)_\p$ denote the image of the Mordell-Weil group via reduction modulo $\p$. We prove in particular that the size of $A(K)_\p$, by varying $\p$, encodes enough information to determine the $K$-isogeny class of $A$, provided that the following necessary condition is satisfied: $B(K)$ has positive rank for every non-trivial abelian subvariety $B$ of $A$. This is the analogue to a result by Faltings of 1983 considering instead the Hasse-Weil zeta function of the special fibers $A_{\p}$.
\end{abstract}


\section{Introduction}

Let $K$ be a number field and $A,A'$ be abelian varieties over $K$.  A well-known result of Faltings (\cite{Faltings83}) implies that $A,A'$ are $K$-isogenous if and only if they have the same $L$-series.  More precisely, if $S=S(A,A')$ is the set of finite primes $\p\subseteq K$ of common good reduction for $A,A'$ and if $S'\subseteq S$ has density one, then $A,A'$ are $K$-isogenous if and only if, for every $\p\in S'$, the special fibers $A_\p,A'_\p$ have the same Hasse--Weil zeta function.  The $L$-series of $A$ is determined, in part, by the function $\nu:\p\in S\mapsto\#A(k_\p)$, and in this paper we consider other functions which one can use to characterize $K$-isogeny.

Let $\G\subseteq A(K),\G'\subseteq A'(K)$ be subgroups, and for each $\p\in S$, let $\G_\p\subseteq A(k_\p)$, $\G'_\p\subseteq A'(k_\p)$ be the respective reductions.  For each prime $\ell$, we consider the composition of the functions $\p\mapsto\G_\p$ and $\p\mapsto\G'_\p$ with the function which sends a finite group $G$ to the $\ell$-adic valuation of the order, exponent, or radical (of the order) of $G$ and which we denote $\ord_\ell(G)$, $\exp_\ell(G)$, and $\rad_\ell(G)$ respectively.  Rather than consider these functions for arbitrary $A,A'$ and $\G,\G'$, we place conditions on $A$ and $\G,\G'$.

We say $A$ is {\it square free} if the only abelian variety $B$ for which there exists a $K$-homomorphism $B^2\to A$ with finite kernel is $B=0$.  We say $\G$ (resp.~$\G'$) is a {\it submodule} if and only if it is an $\End_K(A)$-submodule (resp.~$\End_K(A')$-submodule), and we say $\G$ is {\it dense} if and only if $\pi(\G)\neq\zs$ for every $\pi\neq 0\in\End_K(A)$.

\begin{thm}\label{thm1}
Let $A,A'$ be abelian varieties and $S'\subseteq S(A,A')$ have density one, and suppose $\G\subseteq A(K)$, $\G'\subseteq A'(K)$ are submodules.  If $\G$ is dense and if $\ell\gg 0$, then the following are equivalent:
\begin{enumerate}
\item there exists $\phi\in\Hom_K(A,A')$ such that $\ker(\phi)$ and $[\phi(\G):\phi(\G)\cap\G']$ are finite;
\item $\ord_\ell(\G_\p)\leq \ord_\ell(\G'_\p)$ for every $\p\in S'$.
\end{enumerate}
If moreover $A$ is square free and if $\ell\gg 0$, then these are equivalent to the following:
\begin{enumerate}
\item[3.] $\exp_\ell(\G_\p)\leq \exp_\ell(\G'_\p)$ for every $\p\in S'$;
\item[4.] $\rad_\ell(\G_\p)\leq \rad_\ell(\G'_\p)$ for every $\p\in S'$.
\end{enumerate}
\end{thm}

Clearly if $\G$ is $\zs$ or even finite, then conditions 2,3,4 hold \emph{regardless} of what $A,A',\G'$ are.  In order to avoid pathologies like this assume $\G$ is dense, or equivalently, the intersection of $\G$ with each non-trivial abelian subvariety $B\subseteq A$ is infinite.  Also, for any finite group $G$, $\exp_\ell(G\times G)=\exp_\ell(G)$ and $\rad_\ell(G\times G)=\rad_\ell(G)$, hence the reason we must suppose $A$ is square free in 3~and 4.  

As one might expect, Kummer theory lies at the core of our proof of the theorem, and the subgroups which give the cleanest statements, especially when characterizing when distinct subgroups are `independent,' are submodules.  The basic strategy we employ to prove an equivalence such as $1\Leftrightarrow 2$ is to prove two implications: $1\Rightarrow 2$ and $\neg 1\Rightarrow\neg 2$.  The first implication is straightforward.  A crucial notion which appears in the second implication is of `almost free' points, and we develop this notion in section~\ref{sec:suff_indep}.  Prior to this we give some preliminary results in section~\ref{sec:prelim}, and finally, we prove theorem~\ref{thm1} in section~\ref{sec:proof_thm1}.

Note, the Mordell-Weil group of an abelian variety is a dense submodule if and only if the Mordell-Weil group of every abelian subvariety is infinite. Then we have the following:

\begin{cor}
Let $A,A'$ be abelian varieties and $S'\subseteq S(A,A')$ have density one, and suppose $B(K)$ is infinite for every non-trivial abelian subvariety $B\subseteq A$.  For every fixed $\ell\gg 0$, the $K$-isogeny class of $A$ is determined by the function $\p\in S'\mapsto \#A(K)_\p$.  If moreover $A$ is square-free and $\ell\gg 0$, then the $K$-isogeny class of $A$ is determined by the function $\p\in S'\mapsto \rad_\ell A(K)_\p$, hence a fortiori by the function $\p\in S'\mapsto \exp_\ell A(K)_\p$.
\end{cor}

If $A(K)$ and $A'(K)$ are free of rank $1$ and the functions $\p\mapsto\exp_\ell(A(K)_\p)$ and $\p\mapsto\exp_\ell(A'(K)_\p)$ are considered, the above result relates to the so-called support problem (cf. \cite[thm.~1.2]{DemeyerPerucca}).  Note that there exist pairs of elliptic curves over a number field $K$ which are not $K$-isomorphic but such that for every prime number $\ell$ there is a $K$-isogeny between them of degree coprime to $\ell$ (cf.~\cite[sec.~12]{Zarhin}).  This implies that it is not possible to characterize the $K$-isomorphism class of $A$ by knowing the order and the exponent of $A(K)_{\p}$ for $\p$ varying in a set of density $1$.


\subsection{Notation}

Unless explicitly stated otherwise, we assume all abelian varieties, subvarieties, homomorphisms, etc.~are defined over $K$.  Given an abelian variety $A$, we denote by $S(A)$ the set of finite primes $\p\subset K$ of good reduction for $A$, and we write $k_\p$ for the residue field and $A(k_\p)$ for the group of $k_\p$-rational points.  By the density of a subset $S'\subseteq S(A)$ we mean the Dirichlet density.  We also write $\E(A)$ for the ring $\End_K(A)$, and given a second abelian variety $B$, we write $\H(A,B)$ for $\Hom_K(A,B)$.

Given $\p\in S(A)$ and a subgroup $\G\subseteq A(K)$, we write $\G_\p\subseteq A(k_\p)$ for the reduction of $\G$ modulo $\p$.  Moreover, for each rational prime $\ell$, we define the following functions on $S(A)$:
\[
	\nu_{\ell,\G} : \p\mapsto\ord_\ell(\G_\p),
	\quad
	\ve_{\ell,\G} : \p\mapsto\exp_\ell(\G_\p),
	\quad
	\rho_{\ell,\G} : \p\mapsto\rad_\ell(\G_\p).
\]
They respectively express the $\ell$-adic valuations of the size, the exponent, and the radical of the size of $\G_\p$.


\section{Preliminaries}\label{sec:prelim}

In this section we develop results we need for the proof of theorem~\ref{thm1}.

\subsection{Homomorphisms}

Let $A,B$ be abelian varieties. We make frequent use of the following lemma:

\begin{lem}\label{lem:isogenies}
Let $\phi\in\H(A,B)$ and let $B'\subseteq B$ be the image of $\phi$.  There exist $A'\subseteq A$ and $\psi\in\H(B,A')$ such that 
 $B'+\ker(\psi)=B$ and such that the restrictions $\psi\mid_{B'}$ and $\phi\mid_{A'}$ are isogenies between $A'$ and $B'$ and satisfy $\phi\psi\mid_{B'}=[m]_{B'}$, $\psi\phi\mid_{A'}=[m]_{A'}$ for some $m\geq 1$.
\end{lem}

\begin{proof}
If $A''\subseteq A$ is the kernel of $\phi$, then the Poincar\'e Reducibility Theorem implies there exists $A'\subseteq A$ such that $A'+A''=A$ and $A'\cap A''$ is finite.  Similarly, there exists $B''\subseteq B$ such that $B'+B''=B$ and $n=\#(B'\cap B'')$ is finite, and then the restriction $\phi|_{A'}:A'\to B'$ is isogeny and there exists $\hat\phi:B'\to A'$ and $m'\geq 1$ such that $\hat\phi\phi|_{A'}=[m']_{A'}$ and $\phi\hat\phi=[m']_{B'}$ (cf.~\cite[lem.~A.5.1.5]{HindrySilverman}).  The restriction of $n\hat\phi$ to $B'\cap B''$ is trivial, hence $n\hat\phi$ extends (uniquely) to a homomorphism $\psi:B\to A'$ such that $B''\subseteq\ker(\psi)$ and $\psi\phi|_{A'}=[m]_{A'}$, $\phi\psi|_{B'}=[m]_{B'}$ for $m=m'n$.
\end{proof}

\begin{cor}\label{cor:sym_zs}
$\H(A,B)=\zs$ if and only if $\H(B,A)=\zs$.
\end{cor}

\begin{proof}
The statement is symmetric in $A,B$, so it suffices to suppose $\H(A,B)\neq\zs$ and show that $\H(B,A)\neq\zs$.  If $\phi\neq 0\in\H(A,B)$ and if $\psi\in\H(B,A)$ and $m\geq 1$ are as in lemma~\ref{lem:isogenies}, then $\psi\phi\neq 0$, so $\psi\neq 0$ and hence $\H(B,A)\neq\zs$.
\end{proof}


\subsection{Images of Submodules}

Let $\G\subseteq A(K)$ be a submodule and $\phi\in\H(A,B)$.  We write $\phi_*(\G)\subseteq B(K)$ for the submodule generated by $\phi(\G)$.

\begin{lem}\label{lem:gen_im}
Suppose $\hat\phi\in\H(B,A)$ and $m\geq 1$ satisfy $\hat\phi\phi=[m]_A$.  If $\phi$ is an isogeny, then the index of $\phi(\G)$ in $\phi_*(\G)$ divides $m$.
\end{lem}

\begin{proof}
Suppose $\phi$ is an isogeny and thus $\phi\hat\phi=[m]_B$.  If $P_1,\ldots,P_r\in\G$ and if $\phi_1,\ldots,\phi_r\in\E(B)$, then $Q=\Sigma_i \phi_i\phi(P_i)$ satisfies $mQ=\phi(\Sigma_i \hat\phi\phi_i\phi(P_i))\in\phi(\G)$ and thus $m\phi_*(\G)\subseteq\phi(\G)$.
\end{proof}

\begin{rem}
If $Q_1,\ldots,Q_r\in\phi(G)$ and if $\phi_1,\ldots,\phi_r\in\E(B)$, then for $Q=\Sigma_i\phi_i(Q)$ and $\p\in S(A)\cap S(B)$, the exponent of the reduction $Q_\p$ divides the least-common multiple of the exponents of the reductions $Q_{i,\p}$, thus $\exp_\ell(\phi_*(\G)_\p)=\exp_\ell(\phi(\G)_\p)$ for every $\ell$.
\end{rem}


\subsection{Dense Submodules}

Recall that a submodule $\G\subseteq A(K)$ is {\it dense} if and only if it satisfies condition 3 in the following lemma:

\begin{lem}\label{lem:infinite}
If $\G\subseteq A(K)$ is a submodule, then the following are equivalent:
\begin{enumerate}
\item $\phi(\G)\neq \zs$ for every abelian variety $B$ and $\phi\neq 0\in \H(A,B)$;
\item $\phi'(\G)\neq \zs$ for every 
simple abelian variety $B'$ and $\phi'\neq 0\in \H(A,B')$;
\item $\pi(\G)\neq\zs$ for every $\pi\neq 0\in \E(A)$.
\end{enumerate}
\end{lem}
\begin{proof}
Clearly $1\Rightarrow 2,3$.  Suppose $B$ is an abelian variety and $\phi\neq 0\in\H(A,B)$, and let $B'\subseteq B$ be a non-zero simple abelian subvariety.  If $B'\subseteq \phi(A)$ and if $\pi'\in\H(B,B')$ and $m\geq 1$ satisfy $\pi'|_{B'}=[m]_{B'}$, then $\phi'=\pi'\phi\neq 0\in\H(A,B')$ since the composition of $\phi'$ with inclusion $B'\subseteq B$ equals $m\phi$. 
In particular, if $\phi'(\G)\neq\zs$, then $\phi(\G)\neq\zs$, thus $2\Rightarrow 1$.
Similarly, if $\psi\in\H(B,A)$ and $m\geq 1$ satisfy $\phi\psi|_{\phi(A)}=[m]_{\phi(A)}$, then $\pi=\psi\phi\neq 0\in\E(A)$ since $\phi\pi=m\phi\neq 0$, and thus $3\Rightarrow 1$.
\end{proof}

\begin{rem}\label{rem:infinite}
If $\G\subseteq A(K)$ is a finite submodule, then it is not dense, so a dense submodule is infinite.  Conversely, if $A$ is simple and if $\G\subseteq A(K)$ is an infinite submodule, then $\G$ is dense.
\end{rem}


\subsection{Isogeny Invariance}\label{sec:proof:isogeny}

Let $\G\subseteq A(K)$, $\G'\subseteq A'(K)$ be submodules.  The following lemma shows that condition 1 of theorem~\ref{thm1} is isogeny invariant:

\begin{lem}\label{lem:isogeny}
Suppose that $\iota\in\H(A,B)$, $\iota'\in\H(A',B')$ are isogenies. Then there exist $\hat\iota\in\H(B,A)$, $\hat\iota'\in\H(B',A')$ and $m,m'\geq 1$ satisfying $\hat\iota\iota=[m]_A$ and $\hat\iota'\iota'=[m']_{A'}$, and the following are equivalent:
\begin{enumerate}
\item $\exists\,\phi\in\H(A,A')$ such that $\ker[\phi]$ and $[\phi(m\G):\phi(m\G)\cap m'\G']$ are finite;
\item $\exists\,\phi\in\H(A,A')$ such that $\ker[\phi]$ and $[\phi(\G):\phi(\G)\cap\G']$ are finite;
\item $\exists\,\phi'\in\H(B,B')$ such that $\ker[\phi']$ and $[\phi'\iota(\G):\phi'\iota(\G)\cap\iota'(\G')]$ are finite;
\item $\exists\,\phi'\in\H(B,B')$ such that $\ker[\phi']$ and $[\phi'(\iota_*(\G)):\phi'(\iota_*(\G))\cap\iota'_*(\G')]$ are finite.
\end{enumerate}
\end{lem}
\begin{proof}
Clearly $1\Leftrightarrow 2$. By lemma~\ref{lem:gen_im}, $[\iota_*(\G):\iota(\G)]$ and $[\iota'_*(\G'):\iota'(\G')]$ are finite, and thus $3\Leftrightarrow 4$.  Let $\hat\iota:B\to A$ and $m\geq 1$ satisfy $\hat\iota\iota=[m]_A$, and let $\hat\iota':B'\to A'$ and $m'\geq 1$ be defined similarly (see lemma~\ref{lem:isogenies}).  If $\phi\in\H(A,A')$ has finite kernel, then $\phi'=\iota'\phi\hat\iota$ lies in $\H(B,B')$ and also has finite kernel.  If moreover $\phi(\G)\cap\G'$ has finite index in $\phi(\G)$, then $m\phi(\G)\cap \G'$ has finite index in $m\phi(\G)=\phi\hat\iota(\iota(\G))$ and thus $\iota'\phi\hat\iota(\iota(\G))\cap\iota'(\G')$ has finite index in $\iota'\phi\hat\iota(\iota(\G))=\phi'(\iota(\G))$.  That is, $2\Rightarrow 3$, and a similar argument shows $3\Rightarrow 1$ since $\hat\iota\iota(\G)=m\G$ and $\hat\iota'\iota'(\G')=m'\G'$.
\end{proof}

The following lemma shows that, in theorem~\ref{thm1}, $1\Rightarrow 2$ and moreover $1\Rightarrow 3,4$ if $A$ is square free:

\begin{lem}\label{lem:isogenous}
Let $\phi\in\H(A,A')$.  If $\ker(\phi)$ and $i=[\phi(\G):\phi(\G)\cap\G']$ are finite, then the following holds for $\ell\nmid i\cdot\deg(\phi)$:
\[
	\nu_{\ell,\G}(\p)\leq\nu_{\ell,\G'}(\p),\ \ 
	\ve_{\ell,\G}(\p)\leq\ve_{\ell,\G'}(\p),\ \ 
	\rho_{\ell,\G}(\p)\leq\rho_{\ell,\G'}(\p)\ \quad
	\forall\p\in S(A)\cap S(A').
\]
\end{lem}

\begin{proof}
Let $\p\in S(A)\cap S(A')$.  If $\ell\nmid\deg(\phi)$, then $\phi$ induces an isomorphism of the $\ell$-parts of $\G_\p$ and $\phi(\G)_\p$, thus $\nu_{\ell,\G}(\p)=\nu_{\ell,\phi(\G)}(\p)$.  Moreover, if $\ell\nmid i$, then the $\ell$-parts of $\phi(\G)_\p\cap\G'_\p$ and $\phi(\G)_\p$ coincide, so $\nu_{\ell,\phi(\G)}(\p)=\nu_{\ell,\phi(\G)\cap\G'}(\p)\leq \nu_{\ell,\G'}(\p)$.  Thus the first inequality holds, and a similar argument yields the other inequalities.
\end{proof}

\begin{cor}\label{cor:isogeny}
Suppose $\G\subseteq A(K)$, $\G'\subseteq A'(K)$ are submodules, $\iota\in\H(A,B)$, $\iota'\in\H(A',B')$ are isogenies, and $d\geq 1$ is an integer.  Theorem~\ref{thm1} holds for $A,A',\G,\G'$ and $\ell\nmid d$ if and only if it holds for $B,B',\iota_*(\G),\iota'_*(\G)$ and $\ell\nmid d\cdot\deg(\iota)\cdot\deg(\iota')$.
\end{cor}

\begin{proof}
The equivalences of lemma~\ref{lem:isogeny} implies that condition 1 of theorem~\ref{thm1} holds for $A,A',\G,\G'$ if and only if it holds for $B,B',\iota_*(\G),\iota_*(\G')$.  For the remaining three conditions of theorem~\ref{thm1}, we observe that, for each $\ell\nmid\deg(\iota)$, we have $\nu_{\ell,\G}=\nu_{\ell,\iota_*(\G)}$, $\rho_{\ell,\G}=\rho_{\ell,\iota_*(\G)}$, and $\e_{\ell,\G}=\e_{\ell,\iota_*(\G)}$ on $S(A)\cap S(A')$.  Similarly, on $S(A)\cap S(A')$ we have $\nu_{\ell,\G'}=\nu_{\ell,\iota_*'(\G')}$, $\rho_{\ell,\G'}=\rho_{\ell,\iota_*'(\G')}$, and $\e_{\ell,\G'}=\e_{\ell,\iota_*'(\G')}$, for each $\ell\nmid\deg(\iota')$.  Thus if $\ell\nmid\deg(\iota)\cdot\det(\iota')$, then each condition of theorem~\ref{thm1} holds for both $A,A',\G,\G'$ and $B,B',\iota_*(\G),\iota'_*(\G')$ or for neither.
\end{proof}


\section{Almost Free Points}\label{sec:suff_indep}

Let $A_1,\ldots,A_r$ be abelian varieties.  We say $P_1\in A_1(K),\ldots,P_r\in A_r(K)$ are {\it almost free} points if and only if they have infinite order and the following implication holds for all $i$:
\begin{equation}\label{eqn:si}
	\Pi_j\phi_j\in\Pi_j\H(A_j,A_i),
	\ \ \Sigma_j \phi_j(P_j) = 0
	\quad\Rightarrow\quad
	\phi_1(P_1) = \cdots = \phi_r(P_r) = 0.
\end{equation}
Note, if $X,Y\subset A(K)$ are subsets of almost free points and if $\G,\G'$ are the respective submodules they generate, then non-zero elements of $\G\cap\G'$ correspond bijectively to violations of (\ref{eqn:si}) for $X\cup Y$ because such elements have (exactly)  two representations, one each in elements of $X,Y$ respectively.

Recall that points $P_1,\ldots,P_r\in A(K)$ are {\it independent} (or {\it free}) if and only if $\Sigma_i\phi_i(P_i)=0$ for $\phi_i\in\E(A)$ implies $\phi_i=0$ for all $i$ (cf.~\cite[def.~3 and rem.~6]{Peruccaord1}).  If $P\in A(K)$ is independent, then it is almost free, but the converse does not hold in general.  For example, if $A_1,A_2$ are non-isogenous and simple and if $P_1\in A_1(K),P_2\in A_2(K)$ have infinite order, then $P_1,P_2,P_1+P_2$ are each almost free points of $A=A_1\times A_2$, but only $P_1+P_2$ is free.

\begin{lem}\label{lem:indep}
Suppose $P_1\in A_1(K),\ldots,P_r\in A_r(K)$ are almost free, and for each $i$, suppose the Zariski closure $B_i\subseteq A_i$ of $P_i$ is connected.  If $B=\Pi_i B_i$, then $P=\Pi_i P_i$ is independent.
\end{lem}

\begin{proof}
Suppose $\phi\in\E(B)$.  We may write $\phi$ as an $r\times r$ matrix $(\phi_{ij})$ where $\phi_{ij}\in\H(B_i,B_j)$, and thus $\phi(P)=\Pi_j P'_j$ for $P'_j=\Sigma_i\phi_{ij}(P_i)$.  We must show that if $\phi(P)=0$, then $\phi=0$.  If $\phi(P)=0$, then $P'_j=0$ for all $j$, and hence $\phi_{ij}(P_i)=0$ for all $i,j$ since $P_1,\ldots,P_r$ are almost free.  Therefore, since $P_i$ is Zariski dense in $B_i$ for every $i$, we have $\phi_{ij}=0$ for every $i,j$, that is, $\phi=0$.
\end{proof}

The following proposition is a key ingredient in our proof of theorem~\ref{thm1}:

\begin{prop}\label{prop:indep}
Suppose $P_1,\ldots,P_r\in A(K)$ and $m_1,\ldots,m_r\geq 0$.  If $P_1,\ldots,P_r$ are almost free, then for every $\ell\gg 0$, the following set has positive density:
\[
	S_{\ell,m} :=
		\{\ 
			\p\in S(A) : \e_{\ell,P_i}(\p)=m_i,\forall i
		\ \}.
\]
\end{prop}

\begin{proof}
Let $A_i\subseteq A$ be the Zariski closure of $P_i$.  Up to replacing $P_i$ by $mP_i$, for some $m\geq 1$, and excluding $\ell$ which divide $m$, we suppose that $A_i$ is connected.  Then lemma~\ref{lem:indep} implies $P=\Pi_i P_i$ is independent in $B=\Pi_i A_i$, and thus the proposition follows from \cite[prop.~12]{Peruccaord1}).
\end{proof}

If $A$ is simple and if $\phi\neq 0\in\E(A)$, then the kernel of $\phi$ is finite and so $\phi(P)\neq 0$ for any non-torsion $P\in A(K)$.  Hence if $A$ is simple and if $P\in A(K)$ is almost free, then $P$ is independent.  An analogous remark also holds for more points.  The following lemma gives a mildly different characterization of almost free points when $A$ is simple:

\begin{lem}\label{lem:simple_indep}
Let $A$ be simple and $P_1,\ldots,P_r\in A(K)$ be points of infinite order.  The following are equivalent:
\begin{enumerate}
\item $P_1,\ldots,P_r$ are almost free;
\item if $\phi_1,\ldots,\phi_r\in\E(A)$ satisfy $\Sigma_i\phi_i(P_i)=0$ and if $\phi_1\in\Z$, then $\{\phi_i(P_i)\}=\zs$.
\end{enumerate}
\end{lem}

\begin{proof}
It is clear that $1\Rightarrow 2$ (cf.~(\ref{eqn:si})), so suppose that 2 holds and $\phi_1,\ldots,\phi_r\in\E(A)$ satisfy $\Sigma_i\phi_i(P_i)=0$.
If $\phi_1=0$, then it lies in $\Z$ and hence 2 implies $\phi_1(P_1)=\cdots=\phi_r(P_r)=0$.  If $\phi_1\neq 0$, then its kernel is a proper algebraic subgroup of $A$, hence it must be a finite subgroup, since $A$ is simple, and thus $\phi_1$ is an isogeny.  Let $\psi:A\to A$ and $m\geq 1$ satisfy $\psi\phi_1=[m]_A$.  Then $\Sigma_i\psi\phi_i(P_i)=0$ and $\psi\phi_1=[m]_A\in\Z$, so 2 implies $\{\psi\phi_i(P_i)\}_i=\zs$.  Therefore $\phi_i(P_i)$ lies in the kernel of $\psi$ for every $i$, and thus $\{\phi_i(P_i)\}_i=\zs$ since the points $P_i$ have infinite order and $A$ is simple.
\end{proof}

If $A$ is simple and if $\G\subseteq A(K)$ is a submodule, then one can repeatedly apply the following corollary in order to find a finite-index free-submodule $\G'\subseteq\G$ and an explicit basis of $\G'$:

\begin{cor}\label{cor:free_basis}
Suppose $A$ is simple and $P_1,\ldots,P_r\in A(K)$ are almost free, and let $\G\subseteq A(K)$ be the submodule they generate.  If $Q\in A(K)$ satisfies $nQ\not\in\G$ for every $n\geq 1$, then $P_1,\ldots,P_r,Q$ are almost free.
\end{cor}

\begin{proof}
The points $P_1,\ldots,P_r,Q$ satisfy 2 of lemma~\ref{lem:simple_indep} and thus they are almost free.
\end{proof}

\begin{cor}\label{cor:indep_Q}
Suppose $A$ is simple and $\G\subseteq A(K)$ is a submodule.  If $Q\in A(K)$ satisfies $nQ\not\in\G$, for every $n\geq 1$, and if $\phi\in\E(A)$ satisfies $\phi(Q)\in\G$, then $\phi(Q)$ is torsion.
\end{cor}

\begin{proof}
Let $\{P_1,\ldots,P_r\}\subset\G$ be a maximal subset of almost free points, and let $\G'\subseteq\G$ be the finite-index submodule they generate (cf. cor.~\ref{cor:free_basis}).  If $\phi,\phi_1,\ldots,\phi_r\in\E(A)$ satisfy $\phi(Q)\in\G$ and $\phi(nQ)+\Sigma_i\phi_i(P_i)=0$, where $n=[\G:\G']$, then corollary~\ref{cor:free_basis} implies $\{\phi(nQ),\phi_i(P_i)\}_i=\zs$.  That is, $\phi(nQ)=0$ and thus $\phi(Q)$ is torsion of order dividing $n$.
\end{proof}

For general $A$, we can repeatedly apply the following lemma to choose a `quasi-basis' of a submodule, that is, a maximal subset of almost free points which generate a finite-index submodule:

\begin{lem}\label{lem:basis_extend}
Let $P_1,\ldots,P_r\in A(K)$ be almost free, and suppose the submodule $\G\subseteq A(K)$ they generate is torsion free.  If $\G'\subseteq A(K)$ is a submodule such that $\G\cap\G'$ has infinite index in $\G'$, then there exists $Q\in\G'$ such that $P_1,\ldots,P_r,Q$ are almost free and such that the submodule they generate is torsion free.
\end{lem}

\begin{proof}
Let $B,B'\subseteq A$ be abelian varieties such that $\H(B,B')=\H(B',B)=\zs$ and $B+B'=A$.  Let $\hat\pi:B\to A$ and $\hat\pi':B'\to A$ be the natural inclusions, and suppose $\pi\in\H(A,B)$, $\pi'\in\H(A,B')$, and $m,m'\geq 1$ satisfy $\pi\hat\pi=[m]_B$ and $\pi'\hat\pi'=[m']_{B'}$.

Suppose $B$ is minimal with the property that $\pi(\G\cap\G')$ has infinite index in $\pi(\G')$.  Thus there is a simple abelian subvariety $C\subseteq B$ and an isogeny $\psi:B\to C^e$ for some $e\geq 1$, and we define $\pi_i:A\to C$ to be the composition of $\psi\pi$ with projection onto the $i$th factor and let $\hat\pi_i:C\to A$ and $m_i\geq 1$ be such that $\pi_i\hat\pi_i=[m_i]_C$.  We note that $\ker(\pi)$ has finite index in $\cap_i\ker(\hat\pi_i\pi_i)$ and that the intersection of either with $\ker(\pi')$ is finite.


Let $i$ be such that $\pi_i(\G\cap\G')$ has infinite index in $\pi_i(\G')$.
Therefore there exists $Q'\in\G'$ such that $n\pi_i(Q')\not\in\pi_i(\G)$, for every $n\geq 1$, and we let $Q=\hat\pi_i(\pi_i(Q'))$.  Up to replacing $Q'$ by $nQ'$ for some $n\geq 1$, we suppose without loss of generality that the submodule generated by $P_1,\ldots,P_s,Q$ is torsion free.

From the equality $m_i\pi_j(\G)=\pi_i\hat{\pi}_i\pi_j (\Gamma)$ it follows that $[\pi_j(\G):\pi_i(\G)\cap\pi_j(\G)]$ is finite, for every $j$, and so we have $n\pi_j(Q)\not\in\pi_j(\G)$ for every $n\geq 1$.
Moreover, for $\phi\in\E(A)$ and $\phi'_{ij}=\pi_j\phi\hat\pi_i\in\E(C)$, we have the identity $\pi_j(\phi(Q))=\phi'_{ij}\pi_i(Q')$ and thus corollary~\ref{cor:indep_Q} implies $\pi_j\phi(Q)\in\pi_j(\G)$ only if $\pi_j\phi(Q)$ has finite order.  Up to replacing $Q$ by $nQ$, for some $n\geq 1$, we may suppose that $\pi_j\phi(Q)\in\pi_j(\G)$ only if $\pi_j\phi(Q)=0$.
 
Let $\phi_1,\ldots,\phi_r,\phi\in\E(A)$, and suppose $\Sigma_j\phi_j(P_j)+\phi(Q)=0$.  Our assumptions on $B,B'$ imply $\H(C,B')=\zs$, thus $\pi'\phi\hat\pi_i=0$ and $\pi'\phi(Q)=\pi'\phi\hat\pi_i(\pi_i(Q'))=0$.  That is, $\Sigma_j\pi'\phi_j(P_j)=0$ and hence $\Sigma_j\hat\pi'\pi'\phi_j(P_j)=0$.  Since $P_1,\ldots,P_r$ are almost free, $\{\hat\pi'\pi'\phi_j(P_j)\}_j=\zs$, and thus since $\hat\pi':B'\to A$ is injective, we have $\{\pi'\phi_j(P_j)\}_j=\zs$.  Therefore $\{\phi_j(P_j),\phi(Q)\}_j$ lies in the kernel of $\pi'$.  We will show that it also lies in $\cap_k\ker(\hat\pi_k\pi_k)$.

By assumption, $\Sigma_j\pi_k\phi_j(P_j)+\pi_k\phi(Q)=0$ for every $k$, thus $\pi_k\phi(Q)\in\pi_k(\G)$ and so $\pi_k\phi(Q)=0$.  That is, $\Sigma_j\pi_k\phi_j(P_j)=0$ for every $k$, and thus $\Sigma_j\hat\pi_k\pi_k\phi_j(P_j)=0$ and so $\{\hat\pi_k\pi_k\phi_j(P_j)\}_j=\zs$ since $P_1,\ldots,P_r$ are almost free.  Therefore $\{\phi_j(P_j),\phi(Q)\}_j$ lies in the kernel of $\hat\pi_k\pi_k$, for every $k$.  That is, $\{\phi_j(P_j),\phi(Q)\}_j$ lies in both $\ker(\pi')$ and $\cap_k\ker(\hat\pi_k\pi_k)$ and thus in a finite subgroup of $A(K)$.  Since $\G$ is torsion free and $\phi_j(P_j)\in\G$ for each $j$, we must have $\{\phi_j(P_j)\}_j=\zs$ and thus $\phi(Q)=0$ as well.  That is, $P_1,\ldots,P_r,Q$ are almost free.
\end{proof}


\section{Proof of Theorem~\ref{thm1}}\label{sec:proof_thm1}

We have already seen that if 1 holds, then 2 holds, and if moreover $A$ is square free, then 3 and 4 hold (see section~\ref{sec:proof:isogeny}).  Thus we suppose that 1 fails and show that 2, 3, 4 fail accordingly.

Suppose $P\in A(K)$ and $Q_1,\ldots,Q_r\in A'(K)$ are points, and for each prime $\ell$ and $m\geq 0$, consider the following set:
\[
	S_{\ell,m}(P,Q_1,\ldots,Q_r) :=
		\{\ 
			\p\in S(A)\cap S(A') : \e_{\ell,P}(\p)=m,\ \e_{\ell,Q_i}(\p)=0\mbox{ for  }i=1,\ldots,r
		\ \}.
\]
The basic strategy underlying our proof is to first make judicious choices of $P,Q_1,\ldots,Q_r$ so that we can analyze the $\ell$-parts of $\G_\p,\G'_\p$ for $\ell\gg 0$ and varying $\p$.  In particular, we will show these sets usually have positive density and deduce that 2 fails for $\ell\gg 0$, and moreover, that 3 and 4 fail when $A$ is square free and $\ell\gg 0$.


Let $B'$ be an abelian variety such that $A$ and $A'$ each have some abelian subvariety isogenous to $B'$, and we suppose that $B'$ has maximal dimension.  Then for some abelian subvarieties $B\subseteq A$, $B''\subseteq A'$ we have that $A, A'$ are respectively isogenous to $B\times B'$, $B'\times B''$ so, by corollary~\ref{cor:isogeny}, we may assume that $A=B\times B'$, $A'=B'\times B''$.
We may also assume $\phi\in H(A, A')$ is the composition of the projection on $B'$ and the inclusion $B'\subseteq A'$.

\begin{lem}\label{lem:min_dim}
We have $\H(B,B'')=\H(B'',B)=\zs$.
\end{lem}

\begin{proof}
Suppose $\psi\neq 0\in\H(B,B'')$, and let $\pi:A\to B$ be projection.  Up to composing with the natural embedding $B'' \rightarrow A'$, we can consider the homomorphism $\phi'=\phi+\psi\pi\in\H(A,A')$, and let $C\subseteq A$ be its kernel.  The kernel of $\phi$ is $B=B\times\zs\subseteq A$, and $B'=\zs\times B'\subseteq A$ satisfies $B\cap B'=\zs$, thus the restriction of $\phi'$ to $B'\subseteq A$, which is simply $\phi$, is injective.  The image of $\phi'$ contains the images of $\phi$ and $\psi$, thus it is strictly larger, so $\dim(\ker(\phi'))=\dim(C)<\dim(\ker(\phi))$.  That is, if $\H(B,B'')$ is non trivial, then $\dim(B')$ is not maximal by lemma~\ref{lem:isogenies}.  Finally, $\H(B,B'')=\zs$ if and only if $\H(B'',B)=\zs$ by corollary~\ref{cor:sym_zs}.
\end{proof}

\begin{cor}\label{cor:min_dim}
Suppose $P\in B(K)$, and let $\G_0\subseteq A(K)$ be the submodule it generates.  Then $\phi(\G_0)=\{\psi(P):\psi\in\H(B,A')\}$.
\end{cor}

\begin{proof}
Let $\a\in\E(A)$ and $\a(P)\in\G_0$.  Lemma~\ref{lem:min_dim} implies $\H(B,B'')=\zs$ and thus $\phi\a(P)=\psi(P)$ for some $\psi\in\H(B,B')$.  Conversely, if $\psi\in\H(B,A')$ and if $\pi'':A'\to B''$ is projection, then $\pi''\psi=0\in \H(B,B'')$ and thus $\psi$ factors through the natural embedding $B'\to A'$.  In particular, if we compose with the natural embedding $B'\to A$, then the endomorphism $\a\in\E(A)$ such that $\a|_B=\psi$ and $\a|_{B'}=0$ satisfies $\phi(\a(P))=\psi(P)$ and thus $\psi(P)\in\phi(\G_0)$.
\end{proof}

The following lemma allows us to reduce to the case $B=0$:

\begin{lem}\label{lem:B_pos}
Suppose that condition $1$ in theorem~\ref{thm1} fails. If $B'\neq A$ then 2 fails.  Moreover, if $A$ is square free, then 3 and 4 also fail.
\end{lem}

\begin{proof}
Suppose that $B\neq 0$. If $\pi:A\to B$ is projection, then $\pi(\G)\subseteq B(K)$ is infinite since $\G$ is dense (see lemma~\ref{lem:infinite}), so let $P\in\pi(\G)$ have infinite order and let $\G_0\subseteq\G$ be the submodule it generates.  Lemma~\ref{lem:min_dim} implies $\H(B,B'')=\zs$ and thus $\psi(B)\subseteq B'$ for every $\psi\in\H(B,A')$.  Thus up to replacing $P$ by $nP$ with $n$ the size of the torsion subgroup of $B'(K)$, we suppose without loss of generality that, for every $\psi\in\H(B,A')$, either $\psi(P)=0$ or $\psi(P)$ has infinite order.

Let $\{R_1,\ldots,R_s\}\subset(\phi(\G_0)\cap\G')$ and $\{Q_1,\ldots,Q_r,R_1,\ldots,R_s\}\subseteq\G'$ be maximal subsets of almost free points, and let $\G'_0\subseteq\G'$ be the submodule generated by $Q_1,\ldots,Q_r$ and $\G'_1\subseteq\G'$ be the submodule generated by $R_1,\ldots,R_r$. 
Up to replacing each $Q_i$ by $nQ_i$ with $n$ the size of the torsion subgroup of $B(K)$, we suppose without loss of generality that, for every $i$ and $\psi_i\in\H(A',B)$, either $\psi_i(Q_i)=0$ or $\psi_i(Q_i)$ has infinite order.
Since a maximal subset of almost free points generates a finite index submodule by lemma~\ref{lem:basis_extend}, we may suppose that $\ell$ is coprime with the index of the above submodules.  If $\p\in S_{\ell,m}(P,Q_1,\ldots,Q_r)$, then the $\ell$-part of $\G'_\p$ lies in that of $(\phi(\G_0)\cap\G')_\p$ and thus 
$$\ord_\ell(\G'_\p)\leq\ord_\ell(\G_\p)-\ord_\ell(\G_\p\cap B(K)_\p)\leq\ord_\ell(\G_\p)-m\,.$$
If moreover $A$ is square free, then $\H(B,B')=\zs$ and so $\H(B,A')=\zs$, thus corollary~\ref{cor:min_dim} implies $\phi(\G_0)
=\zs$ and $\e_{\ell,\G'}(\p)=0$.

If $S_{\ell,m}(P,Q_1,\ldots,Q_r)$ has positive density for every $m\geq 0$, then $\nu_{\ell,\G}(\p)-\nu_{\ell,\G'}(\p)$ can be made arbitrarily large on a positive density set, thus 2 fails.  If moreover $A$ is square free, then for every $m\geq 0$, the identities $\e_{\ell,\G}(\p)\geq m$ and $\e_{\ell,\G'}(\p)=0$ hold on a positive density set, thus 3 and 4 fail.  To complete the proof it suffices to show that $P,Q_1,\ldots,Q_r$ are almost free because then proposition~\ref{prop:indep} implies that $S_{\ell,m}(P,Q_1,\ldots,Q_r)$ has positive density for every $\ell\gg 0$ and $m\geq 0$.

The intersection of $\G'_1$ and $\G'_0$ is trivial since $Q_1,\ldots,Q_r,R_1,\ldots,R_s$ are almost free, thus $\phi(\G_0)\cap\G'_0$ is finite.  Suppose $\psi\in\H(B,A')$ and $\psi_1,\ldots,\psi_r\in\E(A')$ satisfy $\psi(P)+\Sigma_i\psi_i(Q_i)=0$.  Then $\psi(P)\in\G'_0$ and corollary~\ref{cor:min_dim} implies $\psi(P)\in\phi(\G_0)$, thus $\psi(P)$ lies in the finite intersection $\phi(\G_0)\cap\G'_0$ and so $\psi(P)=0$ by our assumptions on $P$.  Therefore, since $Q_1,\ldots,Q_r$ are almost free, we have $\{\psi_i(Q_i)\}_i=\zs$ and so $\{\psi(P),\psi_i(Q_i)\}_i=\zs$.

Suppose $\psi\in\E(B)$ and $\psi_1,\ldots,\psi_r\in\H(A',B)$ satisfy $\psi(P)+\Sigma_i\psi_i(Q_i)=0$.  For each $i$, let $\psi'_i\in\H(B,A')$ be such that the kernel of $\psi_i$ has finite index in $\ker(\psi'_i\psi_i)$ (cf.~lemma~\ref{lem:isogenies}), and consider the identity $\psi'_i\psi(P)+\Sigma_j\psi'_i\psi_j(Q_j)=0$.  As in the previous paragraph, $\psi'_i\psi(P)=0$ and thus $\{\psi'_i\psi_j(Q_j)\}_j=\zs$ for every $i$.  That is, $\psi_j(Q_j)$ lies in the kernel of $\psi'_i$ for every $i,j$, and a fortiori for every $i=j$.  In particular, since $\ker(\psi_i)$ has finite index in $\ker(\psi'_i\psi_i)$, $\psi_i(Q_i)$ is torsion.  By our assumptions on $Q_1,\ldots,Q_r$, we have $\{\psi_i(Q_i)\}_i=\zs$ and thus $\{\psi(P),\psi_i(Q_i)\}_i=\zs$.  Therefore, $P,Q_1,\ldots,Q_r$ are almost free as claimed.
\end{proof}

We suppose for the remainder of the section that $B=0$ and thus $A=B'\subseteq A'$ and $\phi$ is the inclusion.  Let $\{Q_1,\ldots,Q_r\}\subset\G'$ be a maximal subset of almost free points, and let $\G'_0\subseteq\G'$ be the finite-index submodule they generate (which we may suppose to be torsion free) and $\ell$ be a prime not dividing $[\G':\G'_0]$.

Suppose condition 1 fails and thus $\G\cap\G'$ has infinite index in $\G$.  Therefore the index of $\phi_*(\G)\cap\G'$ in $\phi_*(\G)$  is infinite and lemma~\ref{lem:basis_extend} implies there exists $P\in\phi_*(\G)$ such that $P,Q_1,\ldots,Q_r$ are almost free.

If $\p\in S_{\ell,m}(P,Q_1,\ldots,Q_r)$, then $\ord_\ell(\G_\p)\geq m$ and $\ord_\ell(\G'_\p)=0$, thus we have the following inequalities:
\[
	\nu_{\ell,\G}(\p) \geq \e_{\ell,\G}(\p)= \e_{\ell,\phi_*(\G)}(\p)\geq m \geq \nu_{\ell,\G'}(\p)=\e_{\ell,\G'}(\p)=0.
\]
In particular, if $\ell\gg 0$, then proposition~\ref{prop:indep} implies $S_{\ell,m}(P,Q_1,\ldots,Q_r)$ has positive density for every $m\geq 0$, and thus 2, 3, and 4 fail.

\bigskip\bigskip\noindent
Q.E.D.


\vspace{0.6cm}

\noindent {\it Chris Hall,} University of Wyoming\\
\noindent \text{E-mail}: chall14@uwyo.edu\\

\noindent {\it Antonella Perucca,} Research Foundation - Flanders (FWO)

\noindent \text{E-mail}: antonellaperucca@gmail.com\\


\begin{thebibliography}{10} \expandafter\ifx\csname url\endcsname\relax   \def\url#1{\texttt{#1}}\fi \expandafter\ifx\csname urlprefix\endcsname\relax\def\urlprefix{URL }\fi

\bibitem{DemeyerPerucca}
J.~Demeyer and A.~Perucca, \emph{The constant of the support problem for abelian varieties}, arXiv:1008.3719.

\bibitem{Faltings83}
G.~Faltings, \emph{Finiteness Theorems for Abelian Varieties over Number Fields}, Arithmetic 
Geometry, Edited by G.~Cornell and J.~H.~Silverman, Springer-Verlag, New York, 1986, 9--27.

\bibitem{HindrySilverman}
M.~Hindry and J.~Silverman, \emph{Diophantine Geometry. An Introduction}, Graduate Texts in Mathematics 201, Springer-Verlag, New York, 2000.

\bibitem{Peruccaord1}
A.~Perucca, \emph{Prescribing valuations of the order of a point in the reductions
  of abelian varieties and tori}, J. Number Theory \textbf{129} (2009), no.~2,
  469--476.

\bibitem{Peruccaord2}  
A.~Perucca, \emph{On the reduction of points on abelian varieties and tori}, Int. Math. Res. Notices \textbf{2011} (2011), no.~7, 293--308.

\bibitem{Zarhin}
Y. Zarhin, \emph{Homomorphisms of abelian varieties over finite fields}, Higher-dimensional geometry over finite fields, Edited by D.~Kaledin and Y.~Tschinkel, IOS, Amsterdam, 2008, 315--343.   

\end{thebibliography}
\end{document}